\documentclass[a4paper,12pt,reqno]{amsart}

\usepackage{amsmath,amssymb,amsthm}
\usepackage{latexsym}
\usepackage{color}
\usepackage{graphicx}
\usepackage{mathrsfs}
\usepackage{enumerate}
\usepackage[abbrev]{amsrefs}
\usepackage[T1]{fontenc}
\usepackage{mathtools}
\mathtoolsset{showonlyrefs=true}

\allowdisplaybreaks[4]

\theoremstyle{plain}
\newtheorem{thm}{Theorem}[section]
\newtheorem{lemm}[thm]{Lemma}

\theoremstyle{definition}

\newtheorem{rem}[thm]{Remark}

\makeatletter

\@addtoreset{equation}{section}
\makeatother

\renewcommand{\div}{\operatorname{div}}
\newcommand{\dB}{\dot{B}}
\newcommand{\supp}{\operatorname{supp}}

\newcommand{\n}[1]{{\left\|#1\right\|}}
\renewcommand{\leq}{\leqslant}

\begin{document}
\title[$2$D stationary Navier--Stokes equation on the whole plane]
{Counter examples for bilinear estimates \\ related to the two-dimensional stationary \\ Navier--Stokes equation}
\author{Mikihiro Fujii}
\address{Institute of Mathematics for Industry Kyushu University, Fukuoka 819--0395, Japan}
\email{fujii.mikihiro.096@m.kyushu-u.ac.jp}
\begin{abstract}
In this paper, we are concerned with bilinear estimates related to the two-dimensional stationary Navier--Stokes equation.
By establishing concrete counter examples, we prove the bilinear estimates fail for almost all scaling critical Besov spaces.
Our result may be closely related to an open problem whether the two-dimensional stationary Navier--Stokes equation on the whole plane $\mathbb{R}^2$ is well-posed or ill-posed in scaling critical Besov spaces.
\end{abstract}
\maketitle

\section{Introduction}\label{sec:intro}
Let $n \geqslant 2$ be an integer and let $1 \leq p,q \leq \infty$.
We consider the following bilinear estimate in the homogeneous Besov space $\dB_{p,q}^{\frac{n}{p}-1}(\mathbb{R}^n)$:
\begin{align}\label{bilin-1}
    \n{(-\Delta)^{-1}\mathbb{P}\div(u\otimes v)}_{\dB_{p,q}^{\frac{n}{p}-1}(\mathbb{R}^n)}
    \leq
    C
    \|u\|_{\dB_{p,q}^{\frac{n}{p}-1}(\mathbb{R}^n)}
    \|v\|_{\dB_{p,q}^{\frac{n}{p}-1}(\mathbb{R}^n)}
\end{align}
for real valued vector fields $u,v \in \dB_{p,q}^{\frac{n}{p}-1}(\mathbb{R}^n)$ satisfying $\div u= \div v= 0$,
with some positive constant $C=C(n,p,q)$ independent of given functions $u$ and $v$.
Here $\mathbb{P}:= I + \nabla\div(-\Delta)^{-1} = \left\{ \delta_{j,k} + \partial_{x_j}\partial_{x_k}(-\Delta)^{-1} \right\}_{1 \leq j,k \leq n}$ denotes the Helmholtz projection onto the divergence free vector fields.
It is proved in \cite{Kan-Koz-Shi-19} that the estimate \eqref{bilin-1} in the higher-dimensional cases $n \geqslant3$ holds for $1 \leq p \leq n$, $1 \leq q \leq \infty$, whereas \cite{Tsu-19-ARMA} showed that it fails for $p=n$, $2<q \leq \infty$ or $n < p \leq \infty$, $1 \leq q \leq \infty$. 
The aim of this paper is to reveal the two-dimensional case is completely different from the higher-dimensional cases and show that the estimate \eqref{bilin-1} with $n=2$ fails for almost all $1 \leq p,q \leq \infty$.

Before we state the main result precisely, we mention the background of this study and the motivation for starting it.
The estimate \eqref{bilin-1} plays a significant role in the analysis of $n$-dimensional stationary Navier--Stokes equation
\begin{align}\label{eq:sns}
    \begin{cases}
    -\Delta u + (u \cdot \nabla)u + \nabla p = f, &\qquad x \in \mathbb{R}^n,\\
    \div u = 0, & \qquad x \in \mathbb{R}^n,
    \end{cases}
\end{align}
where $u=u(x):\mathbb{R}^n \to \mathbb{R}^n$ and $p=p(x):\mathbb{R}^n \to \mathbb{R}$ denote the unknown velocity fields and pressure of the fluid, respectively, whereas $f=f(x):\mathbb{R}^n \to \mathbb{R}^n$ is a given external force.
We note that applying $(-\Delta)^{-1}\mathbb{P}$ to the first equation of \eqref{eq:sns} and using $\mathbb{P} u = u$, $(u \cdot \nabla)u=\div(u \otimes u)$ and $\mathbb{P}(\nabla p) = 0$, 
we see that the stationary Navier--Stokes equation \eqref{eq:sns} is reformulated as
\begin{align}\label{eq:re-sns}
    u=(-\Delta)^{-1}\mathbb{P}f - (-\Delta)^{-1}\mathbb{P}\div(u\otimes u), \qquad x \in \mathbb{R}^n.
\end{align}
In general, it is well-known as the Fujita--Kato principle (see \cite{Fuj-Kat-64}) that it is important to consider the solvability of a partial differential equation in the critical function space with respect to the scaling transform that keeps the equation invariant.
In the case of the stationary Navier--Stokes equation \eqref{eq:re-sns}, if $u$ solves \eqref{eq:re-sns} for some given external force $f$, then the scaled functions
$f_{\lambda}(x):=\lambda^3f(\lambda x)$ and $u_{\lambda}(x):=\lambda u(\lambda x)$ also satisfy \eqref{eq:re-sns} for all $\lambda >0$.
Then, since it holds 
\begin{align}
    \| f_{\lambda} \|_{\dB_{p,q}^{\frac{n}{p}-3}(\mathbb{R}^n)} = \| f \|_{\dB_{p,q}^{\frac{n}{p}-3}(\mathbb{R}^n)},
    \qquad
    \| u_{\lambda} \|_{\dB_{p,q}^{\frac{n}{p}-1}(\mathbb{R}^n)}=\| u \|_{\dB_{p,q}^{\frac{n}{p}-1}(\mathbb{R}^n)}
\end{align}
for all dyadic numbers $\lambda>0$, we see that $f \in \dB_{p,q}^{\frac{n}{p}-3}(\mathbb{R}^n)$ and $u \in \dB_{p,q}^{\frac{n}{p}-1}(\mathbb{R}^n)$ are the scaling critical classes.
Kaneko--Kozono--Shimizu \cite{Kan-Koz-Shi-19} established the bilinear estimate \eqref{eq:re-sns} in the case of $n \geqslant 3$, $1\leq p < n$ and $1 \leq q \leq \infty$.
Making use of the contraction mapping principle via this bilinear estimate,
they \cite{Kan-Koz-Shi-19} considered the well-posedness of \eqref{eq:re-sns} in the scaling critical Besov spaces framework and proved that for $1 \leq p < n$ and $1 \leq q \leq \infty$, \eqref{eq:re-sns} with $n \geqslant 3$ possesses a unique small solution $u \in \dB_{p,q}^{\frac{n}{p}-1}(\mathbb{R}^n)$ provided that  the external force $f \in \dB_{p,q}^{\frac{n}{p}-3}(\mathbb{R}^n)$ is sufficiently small.
On the other hand, Tsurumi \cites{Tsu-19-JMAA,Tsu-19-ARMA} proved that \eqref{eq:re-sns} with $n \geqslant 3$ is ill-posed for $n <p \leq \infty$, $1 \leq q \leq \infty$ and $p=n$, $2<q\leq \infty$ in the sense that the solution map $\dB_{p,q}^{\frac{n}{p}-3}(\mathbb{R}^n)\ni f \mapsto u \in \dB_{p,q}^{\frac{n}{p}-1}(\mathbb{R}^n)$ is discontinuous by constructing counter examples of the bilinear estimate \eqref{bilin-1}.
Recently, Li--Yu--Zhu \cite{Li-Yu-Zhu} investigated the remaining case $p=n$, $1 \leq q\leq 2$.
See \cites{Koz-Yam-95-PJA,Koz-Yam-95-IUMJ,Tsu-19-DIE,Tsu-20,Tsu-23,Tsu-19-N} for other related studies.

In contrast to the higher-dimensional case $\mathbb{R}^n$ with $n \geqslant3$, the question for the well-posedness and ill-posedness of two-dimensional stationary Navier--Stokes equation \eqref{eq:re-sns} with $n=2$ in the scaling critical Besov spaces $(f,u) \in ( \dB_{p,q}^{\frac{2}{p}-3}(\mathbb{R}^2), \dB_{p,q}^{\frac{2}{p}-1}(\mathbb{R}^2) )$ remains as an open problem for all $1 \leq p,q \leq \infty$.
This is because it is quite hard to show the two-dimensional bilinear estimate \eqref{bilin-1} with $n=2$ for any $1 \leq p,q \leq \infty$.
The aim of this paper is to show that the two-dimensional bilinear estimate \eqref{bilin-1} with $n=2$ fails for almost all $1 \leq p,q \leq \infty$.
More precisely, for any $1 \leq p \leq \infty$, $1 \leq q < \infty$ or $2 \leq p \leq \infty$, $q=\infty$, we prove that the following bilinear estimate fails
\begin{align}\label{bilin-2}
    \n{(-\Delta)^{-1}\mathbb{P}\div(u\otimes v)}_{\dB_{p,q}^{\frac{2}{p}-1}(\mathbb{R}^2)}
    \leq
    C
    \|u\|_{\dB_{p,q}^{\frac{2}{p}-1}(\mathbb{R}^2)}
    \|v\|_{\dB_{p,q}^{\frac{2}{p}-1}(\mathbb{R}^2)}
\end{align}
with real valued vector fields $u,v \in \dB_{p,q}^{\frac{n}{p}-1}(\mathbb{R}^2)$ satisfying $\div u= \div v= 0$ and some positive constant $C=C(p,q)$ independent of given functions $u$ and $v$.
Furthermore, for the remaining case $1 \leq p < 2$, $q=\infty$, we show that the following bilinear estimate fails
\begin{align}\label{bilin-3}
    \n{(-\Delta)^{-1}\mathbb{P}\div(u\otimes v)}_{\dB_{p,\infty}^{\frac{2}{p}-1}(\mathbb{R}^2)}
    \leq
    C
    \|u\|_{\dB_{p,\infty}^{\frac{2}{p}-1}(\mathbb{R}^2)}
    \|v\|_{\dB_{p',\infty}^{\frac{2}{p'}-1}(\mathbb{R}^2)}
\end{align}
with real valued vector fields $u \in \dB_{p,\infty}^{\frac{2}{p}-1}(\mathbb{R}^2)$ and $v \in \dB_{p',\infty}^{\frac{2}{p'}-1}(\mathbb{R}^2)$ satisfying $\div u= \div v= 0$
and some positive constant $C=C(p)$ independent of given functions $u$ and $v$.
Here, we note that the estimate \eqref{bilin-3} is slightly stronger than \eqref{bilin-2} due to the embedding $\dB_{p,\infty}^{\frac{2}{p}-1}(\mathbb{R}^2) \hookrightarrow \dB_{p',\infty}^{\frac{2}{p'}-1}(\mathbb{R}^2)$ ($1 \leq p \leq 2$, $1/p+1/p'=1$).

Now, our main result of this paper reads as follows.

\begin{thm}\label{thm}
    Let $p$ and $q$ be as in following cases (i) or (ii):
    \begin{itemize}
        \item [(i)] $1 \leq p \leq \infty$ and $1 \leq q < \infty$,
        \item [(ii)] $2\leq p \leq \infty$ and $q = \infty$.
    \end{itemize}
    Then, there exists a sequence $\{ u_N \}_{N \in \mathbb{N}} \subset \dB_{p,q}^{\frac{2}{p}-1}(\mathbb{R}^2)$ of real valued vector fields on $\mathbb{R}^2$ such that $\div u_N =0$ and
    \begin{gather}
        \lim_{N\to \infty}\| u_N \|_{\dB_{p,q}^{\frac{2}{p}-1}(\mathbb{R}^2)}=0,\\
        \liminf_{N\to \infty} \left\| (-\Delta)^{-1}\mathbb{P}\div(u_N \otimes u_N) \right\|_{\dB_{p,q}^{\frac{2}{p}-1}(\mathbb{R}^2)}>0.
    \end{gather}

    Moreover, for $1 \leq p < 2$ and $q=\infty$,
    there exist two sequences $\{ u_N \}_{N \in \mathbb{N}} \subset \dB_{p,\infty}^{\frac{2}{p}-1}(\mathbb{R}^2)$ and $\{ v_N \}_{N \in \mathbb{N}} \subset \dB_{p',\infty}^{\frac{2}{p'}-1}(\mathbb{R}^2)$ of real valued vector fields on $\mathbb{R}^2$ such that 
    $\div u_N = \div v_N = 0$ and
    \begin{gather}
        \lim_{N\to \infty}\| u_N \|_{\dB_{p,\infty}^{\frac{2}{p}-1}(\mathbb{R}^2)}=\lim_{N\to \infty}\| v_N \|_{\dB_{p',\infty}^{\frac{2}{p'}-1}(\mathbb{R}^2)}=0,\\
        \liminf_{N\to \infty} \left\| (-\Delta)^{-1}\mathbb{P}\div(u_N \otimes v_N) \right\|_{\dB_{p,\infty}^{\frac{2}{p}-1}(\mathbb{R}^2)}>0.
    \end{gather}
    Here, $p':=p/(p-1)$ denotes the H\"older conjugate index of $p$.
\end{thm}
\begin{rem}
Let us mention remarks on Theorem \ref{thm}.
\begin{itemize}
    \item [(1)]
    As a related study of Theorem \ref{thm}, 
    we refer the work of Tsurumi \cite{Tsu-20} where he constructed counter examples for the fractional Leibniz estimates for the product $fg$ of two scalar functions. 
    Compared with \cite{Tsu-20}, the main difficulty of our problem is that we need to estimate more complicated product terms $\mathbb{P}\div (u \otimes v)$ for some divergence free vector fields $u$ and $v$. 
    In \cites{Tsu-19-JMAA,Tsu-19-ARMA}, Tsurumi overcame this problem by finding some vector fields $u$ and $v$ on $\mathbb{R}^n$ with $n \geqslant 3$ for which the bilinear term $\mathbb{P}\div (u \otimes v)$ has a simple structure.
    However, these idea cannot be directly applied to our problem since these good examples $u$ and $v$ are needed at least three components.
    To prove Theorem \ref{thm}, we find that if we choose $u(x)=\nabla^{\perp}\left( \psi(x)\cos(Mx_1) \right)$ with $M \gg 1$ and some smooth function $\psi$ satisfying \eqref{Psi} below, then the leading term of the low frequency part of $\mathbb{P}\div (u \otimes u)$ has a simple structure. 
    This is the key ingredient of our analysis.
    \item [(2)] 
    By our result, we may conjecture that the two-dimensional stationary Navier--Stokes equation \eqref{eq:sns} with $n=2$ is ill-posed in the scaling critical framework $(f,u) \in ( \dB_{p,q}^{\frac{2}{p}-3}(\mathbb{R}^2),\dB_{p,q}^{\frac{2}{p}-1}(\mathbb{R}^2) )$ for all $p$ and $q$ satisfying (i) or (ii) in Theorem \ref{thm}.
    Indeed, following the standard ill-posedness argument proposed in studies such as \cites{Bou-Pav-08,Tsu-19-JMAA,Tsu-19-ARMA,Yon-10}, we see that constructing a sequence of functions which gives a counter example of the bilinear estimate \eqref{bilin-2} is a key ingredient in the proof.
    Theorem \ref{thm} shows that there exists a sequence $\{ f_N\}_{N\in \mathbb{N}} \subset \dB_{p,q}^{\frac{2}{p}-3}(\mathbb{R}^2)$ such that
    \begin{gather}
        \lim_{N\to \infty} \| f_N \|_{\dB_{p,q}^{\frac{2}{p}-3}(\mathbb{R}^2)}=0,\qquad
        \lim_{N\to \infty} \n{u_N^{(1)}}_{\dB_{p,q}^{\frac{2}{p}-1}(\mathbb{R}^2)} = 0,\\
        \liminf_{N\to \infty} \n{u_N^{(2)}}_{\dB_{p,q}^{\frac{2}{p}-1}(\mathbb{R}^2)} > 0,
    \end{gather}
    where $u_N^{(1)}$ and $u_N^{(2)}$ denote the first and second iteration defined by  
    \begin{align}
        u_N^{(1)}:= (-\Delta)^{-1}\mathbb{P}f_N,\qquad
        u_N^{(2)}:=(-\Delta)^{-1}\mathbb{P}\div(u_N^{(1)}\otimes u_N^{(1)}).
    \end{align}
    Then, we formally decompose the solution $ u_N $ to \eqref{eq:sns} with the external force $f_N$ as 
    \begin{align}
        u_N = u_N^{(1)} + u_N^{(2)} + w_N,
    \end{align}
    where the perturbation $w_N$ should solve
    \begin{align}\label{eq:p}
        \begin{split}
        &-\Delta w_N + 
        \mathbb{P}\div 
        \left(
        u_N^{(1)} \otimes u_N^{(2)}
        +
        u_N^{(2)} \otimes u_N^{(1)}
        +
        u_N^{(2)} \otimes u_N^{(2)}\right.\\
        &\quad
        \left.
        +
        u_N^{(1)} \otimes w_N
        +
        u_N^{(2)} \otimes w_N
        +
        w_N \otimes u_N^{(1)}
        +
        w_N \otimes u_N^{(2)}
        +
        w_N \otimes w_N
        \right)
        =0.
        \end{split}
    \end{align}
    To complete the proof of the ill-posedness, it is necessary to show the existence $w_N$ obeying \eqref{eq:p} and establish some estimates for it.
    However, since it is quite hard to find a Banach space $X(\mathbb{R}^2) \subset \mathscr{S}'(\mathbb{R}^2)$ such that 
    \begin{align}
        \n{(-\Delta)^{-1}\mathbb{P}\div(u\otimes v)}_{X(\mathbb{R}^2)}
        \leq
        C
        \|u\|_{X(\mathbb{R}^2)}
        \|v\|_{X(\mathbb{R}^2)}
    \end{align}
    for all $u,v \in X(\mathbb{R}^2)$ with $\div u =\div v =0$, we cannot control the perturbation $w_N$ in any function space and thus 
    the question of the well-posedness and ill-posedness for \eqref{eq:sns} in critical Besov spaces is still a challenging problem. 
    \item [(3)]
    For case of $1 \leq p <2$ and $q=\infty$, we only show the failure of the stronger estimate \eqref{bilin-3} than \eqref{bilin-2} and we cannot conclude whether the estimate \eqref{bilin-2} holds or not in this case.
\end{itemize}
\end{rem}

We prepare notations which is used in this paper.
Throughout this paper, we denote by $c$ and $C$ the constants, which may differ in each line. 
In particular, $C=C(*,...,*)$ denote the constant which depends only on the quantities appearing in parentheses. 
Let $\mathscr{S}(\mathbb{R}^2)$ be the set of all Schwartz functions on $\mathbb{R}^2$ and let $\mathscr{S}'(\mathbb{R}^2)$ be the 
set of all tempered distributions on $\mathbb{R}^2$.
We consider a family $\{\phi_j\}_{j \in \mathbb{Z}}$ of functions in $\mathscr{S}(\mathbb{R}^2)$ satisfying
\begin{align}
    \begin{dcases}
    0 \leq \widehat{\phi_0}(\xi) \leq 1,\\
    \supp \widehat{\phi_0} \subset \{ \xi \in \mathbb{R}^2\ ;\ 2^{-1} \leq |\xi| \leq 2 \},\\
    \widehat{\phi_j}(\xi) = \widehat{\phi_0}(2^{-j}\xi), \\
    \sum_{j \in \mathbb{Z}}
    \widehat{\phi_j}(\xi)
    =1,
    \qquad
    \xi \in \mathbb{R}^2 \setminus \{ 0 \}.
    \end{dcases}
\end{align}
As a related notion, we define
the family $\{ \Delta_j \}_{j \in \mathbb{Z}}$ of Littlewood-Paley frequency localized operators by
$\Delta_j f := \mathscr{F}^{-1}\left[\widehat{\phi_j} \widehat{f}\right]$.
Then, 
we define the homogeneous Besov spaces $\dB_{p,q}^s(\mathbb{R}^2)$ ($1 \leq p,q \leq \infty$, $s \in \mathbb{R}$) by 
\begin{align}
    \dB_{p,q}^s(\mathbb{R}^2)
    :={}&
    \left\{
    f \in \mathscr{S}'(\mathbb{R}^2) / \mathscr{P}(\mathbb{R}^2)
    \ ; \ 
    \| f \|_{\dB_{p,q}^s(\mathbb{R}^2)}
    <
    \infty
    \right\},\\
    \| f \|_{\dB_{p,q}^s(\mathbb{R}^2)}
    :={}&
    \left\|
    \left\{
    2^{sj}
    \| \Delta_j f \|_{L^p(\mathbb{R}^2)}
    \right\}_{j \in \mathbb{Z}}
    \right\|_{\ell^{q}},
\end{align}
where $\mathscr{P}(\mathbb{R}^2)$ is the set of all polynomials on $\mathbb{R}^2$.
In what follows, for a space $X(\mathbb{R}^2)$ of functions on $\mathbb{R}^2$, we use the abbreviation $\| \cdot \|_{X}=\| \cdot \|_{X(\mathbb{R}^2)}$.
We refer to \cites{Bah-Che-Dan-11,Saw-18} for the basic properties of Besov spaces.

This paper is organized as follows.
In Section \ref{sec:key}, we prepare some key lemmas for the proof of Theorem \ref{thm}.
In Section \ref{sec:pf}, we prove Theorem \ref{thm}.

\section{Key lemmas}\label{sec:key}
In this section, we focus on properties of a function $\psi \in \mathscr{S}(\mathbb{R}^2)$ satisfying 
\begin{align}\label{Psi}
    \begin{cases}
    \widehat{\psi}{\rm\ is\ radial\ symmetric},\\
    0 \leq \widehat{\psi}(\xi) \leq 1,\\
    \supp \widehat{\psi} \subset \{ \xi \in \mathbb{R}^2\ ;\ |\xi|\leq 2 \},\\
    \widehat{\psi}(\xi)=1 \quad {\rm for\ all\ }\xi \in \mathbb{R}^2{\rm\ with\ }|\xi|\leq 1.
    \end{cases}
\end{align}
and establish three lemmas, which are key ingredients of the proof of Theorem \ref{thm}.
\begin{lemm}\label{lemm:direct}
Let $\psi \in \mathscr{S}(\mathbb{R}^2)$ satisfy 
\eqref{Psi}
and let 
\begin{align}
    \varphi(x):=\psi(x)\cos(Mx_1),
\end{align}
where $M\geqslant 10$ is a constant.
Then, it holds
\begin{align}\label{di}
    \Delta_j \div (\nabla^{\perp}\varphi \otimes \nabla^{\perp}\varphi)
    =
    \frac{M^2}{2}
    \Delta_j  
    \begin{pmatrix}
        0 \\ \partial_{x_2}(\psi^2)
    \end{pmatrix}
    +
    \frac{1}{2}
    \Delta_j 
    \div ( \nabla^{\perp} \psi \otimes \nabla^{\perp} \psi )
\end{align}
for all $j \in \mathbb{Z}$ with $j \leq 0$.
Here, we have set $\nabla^{\perp}:=(-\partial_{x_2},\partial_{x_1})$.
\end{lemm}
\begin{rem}
    As $M \gg 1$, we regard the first term of the right hand side of \eqref{di} is the leading term.
    This leading term is so simple that we may obtain the appropriate lower bound estimate. See Lemma \ref{lemm:lower-bound} for details.
\end{rem}
\begin{proof}[Proof of Lemma \ref{lemm:direct}]
Let $j \in \mathbb{Z}$ satisfy $j \leq 0$.
By the direct calculations, we see that 
\begin{align}\label{direct-1}
    \div( \nabla^{\perp}\varphi \otimes \nabla^{\perp}\varphi )
    ={}&
    \begin{pmatrix}
        \partial_{x_2}\varphi\partial_{x_1x_2}\varphi - \partial_{x_1}\varphi\partial_{x_2x_2}\varphi \\
        \partial_{x_1}\varphi\partial_{x_1x_2}\varphi - \partial_{x_2}\varphi\partial_{x_1x_1}\varphi
    \end{pmatrix}
\end{align}
and 
\begin{align}
    &
    \partial_{x_1}\varphi(x)
    =
    \partial_{x_1}\psi(x)\cos(Mx_1) - M \psi(x) \sin(Mx_1),\\
    &
    \partial_{x_2}\varphi(x)
    =
    \partial_{x_2}\psi(x)\cos(Mx_1),\\
    &
    \partial_{x_1x_1}\varphi(x)
    =
    \partial_{x_1x_1}\psi(x)\cos(Mx_1) - 2 M\partial_{x_1}\psi(x)\sin(Mx_1) - M^2 \psi(x) \cos(Mx_1),\\
    &
    \begin{aligned}
    \partial_{x_1x_2}\varphi(x)
    ={}&
    \partial_{x_2x_1}\varphi(x)\\
    ={}&
    \partial_{x_1x_2}\psi(x)\cos(Mx_1)
    -
    M\partial_{x_2}\psi(x)\sin(Mx_1),
    \end{aligned}\\
    &
    \partial_{x_2x_2}\varphi(x)
    =
    \partial_{x_2x_2}\psi(x)\cos(Mx_1).
\end{align}
By making use of the elementary properties of the trigonometric functions, we have
\begin{align}
    \partial_{x_2}\varphi(x)\partial_{x_1x_2}\varphi(x)
    =
    \frac{1}{2}
    \partial_{x_2}\psi(x)\partial_{x_1x_2}\psi(x)
    &+
    \frac{1}{2}
    \partial_{x_2}\psi(x)\partial_{x_1x_2}\psi(x)
    \cos(2Mx_1)\\
    &-
    \frac{M}{2}
    (\partial_{x_2}\psi(x))^2
    \sin(2Mx_1).
\end{align}
Here, since the supports of the Fourier transforms of the second and third terms of the above right hand side are included in $\{ \xi \in \mathbb{R}^2\ ;\ 2M-4 \leq |\xi| \leq 2M+4 \}$, they vanish by applying $\Delta_j$.
Thus, we have
\begin{align}
    \Delta_j\left(\partial_{x_2}\varphi\partial_{x_1x_2}\varphi\right)
    =
    \frac{1}{2}
    \Delta_j\left(\partial_{x_2}\psi\partial_{x_1x_2}\psi\right).
\end{align}
Similarly, we obtain
\begin{align}
    &
    \Delta_j\left(\partial_{x_1}\varphi\partial_{x_2x_2}\varphi\right)
    =
    \frac{1}{2}
    \Delta_j\left(\partial_{x_1}\psi\partial_{x_2x_2}\psi\right),\\
    &
    \Delta_j\left(\partial_{x_1}\varphi\partial_{x_1x_2}\varphi\right)
    =
    \frac{M^2}{2}
    \Delta_j\left(\psi\partial_{x_2}\psi\right)
    +
    \frac{1}{2}
    \Delta_j\left(\partial_{x_1}\psi\partial_{x_1x_2}\psi\right),\\
    &
    \Delta_j\left(\partial_{x_2}\varphi\partial_{x_1x_1}\varphi\right)
    =
    -
    \frac{M^2}{2}
    \Delta_j\left(\psi\partial_{x_2}\psi\right)
    +
    \frac{1}{2}
    \Delta_j\left(\partial_{x_2}\psi\partial_{x_1x_1}\psi\right).
\end{align}
Hence, it follows from the above calculations and \eqref{direct-1} that
\begin{align}
    \Delta_j\div( \nabla^{\perp}\varphi \otimes \nabla^{\perp}\varphi )
    ={}&
    \frac{M^2}{2}
    \Delta_j
    \begin{pmatrix}
        0 \\ 2\psi\partial_{x_2}\psi
    \end{pmatrix}\\
    &
    +
    \frac{1}{2}
    \Delta_j
    \begin{pmatrix}
        \partial_{x_2}\psi\partial_{x_1x_2}\psi - \partial_{x_1}\psi\partial_{x_2x_2}\varphi \\
        \partial_{x_1}\psi\partial_{x_1x_2}\psi - \partial_{x_2}\psi\partial_{x_1x_1}\varphi
    \end{pmatrix}\\
    ={}&
    \frac{M^2}{2}
    \Delta_j  
    \begin{pmatrix}
        0 \\ \partial_{x_2}(\psi^2)
    \end{pmatrix}
    +
    \frac{1}{2}
    \Delta_j 
    \div ( \nabla^{\perp} \psi \otimes \nabla^{\perp} \psi )
\end{align}
and we complete the proof.
\end{proof}

\begin{lemm}\label{lemm:direct2}
    For given $N \in \mathbb{N}$, $a_{10},...,a_{N+10}, b_{10},...,b_{N+10} \in \mathbb{R}$ and $\psi \in \mathscr{S}(\mathbb{R}^2)$ satisfying \eqref{Psi},
    we define
    \begin{align}
        u_N
        :=
        \sum_{j=10}^{N+10}
        a_j
        \nabla^{\perp}
        \big(\psi(x)\cos(2^{j^2}x_1)\big),\qquad
        v_N
        :=
        \sum_{j=10}^{N+10}
        b_j
        \nabla^{\perp}
        \big(\psi(x)\cos(2^{j^2}x_1)\big),
    \end{align}
    where $\nabla^{\perp}:=(-\partial_{x_2},\partial_{x_1})$.
    Then, it holds
    \begin{align}
    \Delta_k \div (u_N \otimes v_N)
    ={}&
    \frac{1}{2}
    \sum_{j=10}^{N+10}
    2^{2j^2}
    a_jb_j
    \Delta_k  
    \begin{pmatrix}
        0 \\ \partial_{x_2}(\psi^2)
    \end{pmatrix}\\
    &+
    \frac{1}{2}
    \sum_{j=10}^{N+10}
    a_jb_j
    \Delta_k
    \div ( \nabla^{\perp} \psi \otimes \nabla^{\perp} \psi )
    \end{align}
    for all $k \in \mathbb{Z}$ with $k \leq 0$.
\end{lemm}

\begin{proof}
Let $\varphi_j(x):=\psi(x)\cos(2^{j^2}x_1)$.
By the definition of $u_N$ and $v_N$, we see that
\begin{align}
    \div (u_N \otimes v_N)
    ={}&
    \sum_{j=10}^{N+10}
    a_jb_j
    \div (\nabla^{\perp}\varphi_j \otimes \nabla^{\perp}\varphi_j)\\
    &+
    \sum_{\substack{10 \leq j,\ell \leq N+10 \\ j \neq \ell}}
    a_jb_{\ell}
    \div (\nabla^{\perp}\varphi_j \otimes \nabla^{\perp}\varphi_{\ell})
\end{align}
Since there holds for integers $10 \leq j,\ell \leq N+10$ with $j \neq \ell$,
\begin{align}
    \supp 
    \mathscr{F}
    \left[ \div (\nabla^{\perp}\varphi_j \otimes \nabla^{\perp}\varphi_{\ell}) \right]
    ={}&
    \supp 
    ( \widehat{\varphi_j} \otimes \widehat{\varphi_{\ell}} )\\
    \subset{}&
    \overline{ \supp \widehat{\varphi_j} + \supp \widehat{\varphi_{\ell}} }\\
    ={}&
    \left\{ \xi=\xi^{(j)}+\xi^{(\ell)}\ ;\ 
    \begin{aligned}
    &
    2^{j^2}-2\leq|\xi^{(j)}| \leq 2^{j^2}+2,\\
    &
    2^{\ell^2}-2 \leq |\xi^{(\ell)}| \leq 2^{\ell^2}+2
    \end{aligned}
    \right\}\\
    \subset{}&
    \left\{
    \xi \in \mathbb{R}^2 \ ; \ 
    |\xi| \geqslant 2^{\max\{j,\ell\}^2}-2^{\min\{j,\ell\}^2}-4
    \right\}\\
    \subset{}&
    \left\{
    \xi \in \mathbb{R}^2 \ ; \ 
    |\xi| \geqslant 2^{11^2}-2^{10^2}-4
    \right\},
\end{align}
we have
\begin{align}
    \sum_{\substack{10 \leq j,\ell \leq N+10 \\ j \neq \ell}}a_jb_{\ell}\Delta_k\div (\nabla^{\perp}\varphi_j \otimes \nabla^{\perp}\varphi_{\ell})=0.
\end{align}
Hence, from Lemma \ref{lemm:direct}, we obtain
\begin{align}
    \Delta_k\div (u_N \otimes v_N)
    ={}&
    \sum_{j=10}^{N+10}
    a_jb_j
    \Delta_k\div (\nabla^{\perp}\varphi_j \otimes \nabla^{\perp}\varphi_j)\\
    ={}&
    \frac{1}{2}
    \sum_{j=10}^{N+10}
    2^{2j^2}
    a_jb_j
    \Delta_k  
    \begin{pmatrix}
        0 \\ \partial_{x_2}(\psi^2)
    \end{pmatrix}\\
    &+
    \frac{1}{2}
    \sum_{j=10}^{N+10}
    a_jb_j
    \Delta_k
    \div ( \nabla^{\perp} \psi \otimes \nabla^{\perp} \psi ),
\end{align}
which completes the proof.

\end{proof}

\begin{lemm}\label{lemm:lower-bound}
    Let $\psi \in \mathscr{S}(\mathbb{R}^2)$ satisfy \eqref{Psi}.
    Then, for any $1 \leq p \leq \infty$, there exists a positive constant $c=c(p)$ such that
    \begin{align}
        2^{(\frac{2}{p}-1)j}\n{\Delta_j(-\Delta)^{-1}\mathbb{P}
        \begin{pmatrix}
            0 \\ \partial_{x_2}(\psi^2)
        \end{pmatrix}
        }_{L^p}
        \geqslant c
    \end{align}
    for all $j \in \mathbb{Z}$ with $j \leq -2$.
\end{lemm}

\begin{proof}
Let $j \in \mathbb{Z}$ satisfy $j \leq -2$.
We first consider the case $1 \leq p \leq 2$.
Let
\begin{align}
    A_j
    :=
    \left\{
    \xi \in \mathbb{R}^2 \ ; \ 
    2^{j-1} \leq |\xi| \leq 2^{j+1},\
    \frac{|\xi|}{2} \leq |\xi_2| \leq \frac{|\xi|}{\sqrt{2}} 
    \right\}.
\end{align}
Then, we have
\begin{align}
    &
    2^{(\frac{2}{p}-1)j}
    \n{\Delta_j (-\Delta)^{-1}
    \mathbb{P}
        \begin{pmatrix}
            0 \\ \partial_{x_2}(\psi^2)
        \end{pmatrix}
    }_{L^p}\\
    &\quad
    =
    2^{(\frac{2}{p}-1)j}
    \n{   
    \begin{pmatrix}
            0 \\ \Delta_j \left(1+\partial_{x_2}^2(-\Delta)^{-1}\right) (-\Delta)^{-1}\partial_{x_2}(\psi^2)
    \end{pmatrix}
    }_{L^p}\\
    &\quad
    \geqslant{}
    c
    \n{\Delta_j \left(1+\partial_{x_2}^2(-\Delta)^{-1}\right)(-\Delta)^{-1}\partial_{x_2}(\psi^2)}_{L^2}\\
    &\quad={}
    c
    \n{ \widehat{\phi}(2^{-j}\xi)\left( 1 - \frac{\xi_2^2}{|\xi|^2} \right)\frac{i\xi_2}{|\xi|^2}(\widehat{\psi}*\widehat{\psi})(\xi) }_{L^2}\\
    &\quad
    \geqslant{}
    c2^{-j}
    \n{ \widehat{\phi_0}(2^{-j}\xi)(\widehat{\psi}*\widehat{\psi})(\xi) }_{L^2(A_j)}\\
    &\quad={}
    c\n{ \widehat{\phi_0}(\widetilde{\xi})(\widehat{\psi}*\widehat{\psi})(2^j\widetilde{\xi}) }_{L^2(A_0)}\qquad(\widetilde{\xi}=2^{-j}\xi).
\end{align}
Since $\widehat{\psi}(2^j\widetilde{\xi}-\eta) = \widehat{\psi}(\eta) = 1$ for all $\widetilde{\xi} \in A_0$ and $\eta \in \mathbb{R}^2$ with $|\eta|\leq 1/2$, we see that
\begin{align}
    (\widehat{\psi}*\widehat{\psi})(2^j\widetilde{\xi})
    ={}&\int_{|\eta|\leq 2} \widehat{\psi}(2^j\widetilde{\xi}-\eta) \widehat{\psi}(\eta)d\eta\\
    \geqslant{}&
    \int_{|\eta|\leq \frac{1}{2}} \widehat{\psi}(2^j\widetilde{\xi}-\eta) \widehat{\psi}(\eta)d\eta\\
    ={}&\int_{|\eta| \leq \frac{1}{2}} d\eta
    =\frac{\pi}{4},
\end{align}
which implies
\begin{align}
    \n{\Delta_j (-\Delta)^{-1}\partial_{x_2}(\psi^2)}_{L^2} 
    \geqslant 
    c\n{ \widehat{\phi_0} }_{L^2(A_0)}>0
\end{align}

Next, we consider the case of $2 \leq p \leq \infty$.
A portion of the following calculations are based on \cite{Iwa-Oga-21}*{Proposition 3.1}. 
Using the Bernstein and H\"older inequalities and the Plancherel theorem, it holds
\begin{align}
    &2^{(\frac{2}{p}-1)j}
    \n{\Delta_j (-\Delta)^{-1}
    \mathbb{P}
        \begin{pmatrix}
            0 \\ \partial_{x_2}(\psi^2)
        \end{pmatrix}
    }_{L^p}\\
    &\quad
    =
    2^{(\frac{2}{p}-1)j}
    \n{   
    \begin{pmatrix}
            0 \\ \Delta_j \left(1+\partial_{x_2}^2(-\Delta)^{-1}\right) (-\Delta)^{-1}\partial_{x_2}(\psi^2)
    \end{pmatrix}
    }_{L^p}\\
    &\quad\geqslant
    c
    2^{(\frac{2}{p}-1)j}
    \left\| \Delta_j\left( 1 + \partial_{x_2}^2(-\Delta)^{-1} \right)(-\Delta)^{-\frac{3}{2}}\partial_{x_2}^2(\psi^2) \right\|_{L^p}\\
    &\quad\geqslant
    c
    \left\| e^{-2^{2j}|x|^2}\left(\Delta_j\left( 1 + \partial_{x_2}^2(-\Delta)^{-1} \right)(-\Delta)^{-\frac{3}{2}}\partial_{x_2}^2(\psi^2) \right) \right\|_{L^2}\\
    &\quad=
    c
    \left\|e^{2^{2j}\Delta}\left(\mathscr{F}\left[\Delta_j\left( 1 + \partial_{x_2}^2(-\Delta)^{-1} \right)(-\Delta)^{-\frac{3}{2}}\partial_{x_2}^2(\psi^2) \right]\right)(\xi) \right\|_{L^2(\mathbb{R}^2_{\xi})}.
\end{align}
Since it holds 
\begin{align}
    &\left|e^{2^{2j}\Delta}\left(\mathscr{F}\left[\Delta_j\left( 1 + \partial_{x_2}^2(-\Delta)^{-1} \right)(-\Delta)^{-\frac{3}{2}}\partial_{x_2}^2(\psi^2)\right]\right)(\xi)\right|\\
    &\quad 
    =
    \frac{1}{4\pi2^{2j}}
    \int_{\mathbb{R}^2}
    e^{-\frac{|\xi-\eta|^2}{4\cdot2^{2j}}}
    \widehat{\phi_0}(2^{-j}\eta)
    \left( 1 -\frac{\eta_2^2}{|\eta|^2} \right)
    \frac{\eta_2^2}{|\eta|^3} (\widehat{\psi}*\widehat{\psi})(\eta)d\eta\\
    &\quad\geqslant
    c
    2^{-3j}
    \int_{A_j}
    e^{-2^{-2j-2}|\xi-\eta|^2}
    \widehat{\phi_0}(2^{-j}\eta)(\widehat{\psi}*\widehat{\psi})(\eta)d\eta\\
    &\quad=
    c
    2^{-j}
    \int_{A_0}
    e^{-2^{-2}|2^{-j}\xi-\widetilde{\eta}|^2}
    \widehat{\phi_0}(\widetilde{\eta})(\widehat{\psi}*\widehat{\psi})(2^j\widetilde{\eta})d\widetilde{\eta} \qquad (\widetilde{\eta}=2^{-j}\eta)\\
    &\quad 
    \geqslant
    c
    2^{-j}
    \int_{A_0}
    e^{-2^{-2}|2^{-j}\xi-\widetilde{\eta}|^2}
    \widehat{\phi_0}(\widetilde{\eta})d\widetilde{\eta}
    \\
    &\quad 
    \geqslant
    c
    2^{-j}
    \int_{A_0}
    \widehat{\phi_0}(\widetilde{\eta})d\widetilde{\eta}
    =
    c
    2^{-j}
\end{align}
for all $\xi$ with $2^{j-1} \leq |\xi| \leq 2^{j+1}$,
we have
\begin{align}
    \begin{split}
    &2^{(\frac{2}{p}-1)j}
    \n{\Delta_j (-\Delta)^{-1}
    \mathbb{P}
        \begin{pmatrix}
            0 \\ \partial_{x_2}(\psi^2)
        \end{pmatrix}
    }_{L^p}\\
    &\quad 
    \geqslant
    c
    \left\|e^{2^{2j}\Delta}\left(\mathscr{F}\left[\Delta_j\left( 1 + \partial_{x_2}^2(-\Delta)^{-1} \right)(-\Delta)^{-\frac{3}{2}}\partial_{x_2}^2(\psi^2) \right]\right)(\xi) \right\|_{L^2(2^{j-1} \leq |\xi| \leq 2^{j+1})}\\
    &\quad 
    \geqslant{}
    c2^{-j}
    \left(
    \int_{2^{j-1} \leq |\xi| \leq 2^{j+1}} d\xi
    \right)^{\frac{1}{2}}\\
    &\quad 
    =
    c.
    \end{split}
\end{align}
This completes the proof.
\end{proof}

\section{Proof of Theorem \ref{thm}}\label{sec:pf}
Now, we are in a position to present the proof of Theorem \ref{thm}.
The proof is separated into three parts.
\subsection*{(i) The case of $1 \leq p \leq \infty$ and $1 \leq q < \infty$}
Let $\psi \in \mathscr{S}(\mathbb{R}^2)$ satisfy \eqref{Psi} and let
\begin{align}\label{UN}
    u_N(x) := \frac{1}{N^{\frac{1}{2q}}} \nabla^{\perp}\left(\psi(x) \cos (Mx_1)\right),
\end{align}
where $M \geqslant 10$ is a constant to be determined later.
We note that $u_N$ is a divergence-free real valued vector field.
Since $\supp \widehat{u_N} \subset \{ \xi \in \mathbb{R}^2 \ ; \ M-2 \leq |\xi| \leq M+2 \}$,
we see that
\begin{align}\label{UN0}
    \| u_N \|_{\dB_{p,q}^{\frac{2}{p}-1}}
    \leq
    \frac{C}{N^{\frac{1}{2q}}}M^{\frac{2}{p}}\| \psi \|_{L^p}
    \to 0
    \qquad
    {\rm as\ }N \to \infty.
\end{align}
By Lemma \ref{lemm:direct}, we have
\begin{align}
    \Delta_j(-\Delta)^{-1}\mathbb{P}\div(u_N \otimes u_N)
    &={}
    \frac{M^2}{2N^{\frac{1}{q}}}
    \Delta_j
    (-\Delta)^{-1}
    \mathbb{P}
    \begin{pmatrix}
        0 \\ \partial_{x_2}(\psi^2)
    \end{pmatrix}\\
    &\quad+
    \frac{1}{2N^{\frac{1}{q}}}
    \Delta_j(-\Delta)^{-1}\mathbb{P}\div(\nabla^{\perp}\psi \otimes \nabla^{\perp}\psi)\\
    &=:{}
    \Delta_jI_N^{(1)}
    +
    \Delta_jI_N^{(2)}.
\end{align}
We see by Lemma \ref{lemm:lower-bound} that
\begin{align}
    2^{(\frac{2}{p}-1)j}\n{\Delta_j I_N^{(1)}}_{L^p} \geqslant \frac{cM^2}{N^{\frac{1}{q}}}, \qquad j \leq -2,
\end{align}
which yields
\begin{align}\label{A2-2}
    \left\{\sum_{-N-2 \leq j \leq -2}
    \left(
    2^{(\frac{2}{p}-1)j}
    \n{\Delta_j I_N^{(1)}}_{L^p}
    \right)^q
    \right\}^{\frac{1}{q}}
    \geqslant
    cM^2.
\end{align}
It follows from the Bernstein inequality and the embedding $L^1(\mathbb{R}^2) \hookrightarrow \dB_{p,\infty}^{\frac{2}{p}-2}(\mathbb{R}^2)$ that
\begin{align}\label{A2-3}
    \left\{\sum_{-N-2 \leq j \leq -2}
    \left(
    2^{(\frac{2}{p}-1)j}
    \n{\Delta_j I_N^{(2)}}_{L^p}
    \right)^q
    \right\}^{\frac{1}{q}}
    \leq{}&
    C
    \| \nabla^{\perp}\psi \otimes \nabla^{\perp}\psi \|_{\dB_{p,\infty}^{\frac{2}{p}-2}}\\
    \leq{}&
    C
    \| \nabla^{\perp}\psi \otimes \nabla^{\perp}\psi \|_{L^1}\\
    \leq{}&
    C\|\psi\|_{\dot{H}^1}^2.
\end{align}
Hence, we obtain  by \eqref{A2-2} and \eqref{A2-3} that
\begin{align}
    &
    \left\| (-\Delta)^{-1}\mathbb{P}\div(u_N \otimes u_N) \right\|_{\dB_{p,q}^{\frac{2}{p}-1}}\\
    &\quad 
    \geqslant
    \left\{\sum_{-N-2 \leq j \leq -2}
    \left(
    2^{(\frac{2}{p}-1)j}
    \n{\Delta_j I_N^{(1)} + \Delta_j  I_N^{(2)} }_{L^p}
    \right)^q
    \right\}^{\frac{1}{q}}\\
    &\quad 
    \geqslant
    \left\{\sum_{-N-2 \leq j \leq -2}
    \left(
    2^{(\frac{2}{p}-1)j}
    \n{\Delta_j I_N^{(1)}}_{L^p}
    \right)^q
    \right\}^{\frac{1}{q}}
    -
    \left\{\sum_{-N-2 \leq j \leq -2}
    \left(
    2^{(\frac{2}{p}-1)j}
    \n{\Delta_j I_N^{(2)}}_{L^p}
    \right)^q
    \right\}^{\frac{1}{q}}\\
    &\quad 
    \geqslant
    c_0M^2-C_0\|\psi\|_{\dot{H}^1}^2
\end{align}
for some positive constants $c_0=c_0(p,q)$ and $C_0=C_0(p,q)$.
Then, choosing $M$ so large that $c_0M^2-C_0\|\psi\|_{\dot{H}^1}^2 \geqslant c_0$, we complete the proof of this case.

\subsection*{(ii) The case of $2 < p \leq \infty$ and $q = \infty$.}
Let $\psi \in \mathscr{S}(\mathbb{R}^2)$ satisfy \eqref{Psi} and let
\begin{align}
    u_N(x) := \frac{1}{N} \nabla^{\perp}\bigg(\psi(x) \cos (Nx_1)\bigg),
\end{align}
Then, by the similar calculation as in the previous step (i), 
we see that
\begin{align}
    &\| u_N \|_{\dB_{p,\infty}^{\frac{2}{p}-1}} 
    \leq 
    \frac{C}{N}N^{\frac{2}{p}}\| \psi \|_{L^p}
    \leq 
    \frac{C}{N^{1-\frac{2}{p}}}
\end{align}
for $N\geqslant 10$.
By Lemma \ref{lemm:direct}, we may decompose the bilinear term as
\begin{align}
    \Delta_j(-\Delta)^{-1}\mathbb{P}\div(u_N \otimes u_N)
    &={}
    \frac{1}{2}
    \Delta_j
    (-\Delta)^{-1}
    \mathbb{P}
    \begin{pmatrix}
        0 \\ \partial_{x_2}(\psi^2)
    \end{pmatrix}\\
    &\quad+
    \frac{1}{2N^2}
    \Delta_j(-\Delta)^{-1}\mathbb{P}\div(\nabla^{\perp}\psi \otimes \nabla^{\perp}\psi)
\end{align}
Here, we see by Lemma \ref{lemm:lower-bound} that
\begin{align}\label{stp2-1}
    &
    \inf_{k\leq -2} 
    2^{(\frac{2}{p}-1)k}
    \n{\Delta_k(-\Delta)^{-1}\mathbb{P}
    \begin{pmatrix}
        0 \\ \partial_{x_2}(\psi^2)
    \end{pmatrix}
    }_{L^p}
    \geqslant c.
\end{align}
It follows from the Bernstein inequality and the embedding $L^1(\mathbb{R}^2) \hookrightarrow \dB_{p,\infty}^{\frac{2}{p}-2}(\mathbb{R}^2)$ that
\begin{align}\label{stp2-2}
    \n{(-\Delta)^{-1}\mathbb{P}\div( \nabla^{\perp} \psi \otimes \nabla^{\perp}\psi)}_{\dB_{p,\infty}^{\frac{2}{p}-1}}
    \leq
    C
    \n{\nabla^{\perp}\psi \otimes \nabla^{\perp}\psi}_{L^1}
    \leq
    C\| \psi \|_{\dot{H}^1}^2.
\end{align}
Thus, we obtain
\begin{align}
    \left\| (-\Delta)^{-1}\mathbb{P}\div(u_N \otimes u_N) \right\|_{\dB_{p,\infty}^{\frac{2}{p}-1}}
    \geqslant{}
    c-\frac{C\| \psi \|_{\dot{H}^1}^2}{N^2}
\end{align}
for $N\geqslant 10$,
which completes the proof of this case.

\subsection*{(iii) The case of $1 \leq p \leq 2$ and $q=\infty$}
Let $\psi \in \mathscr{S}(\mathbb{R}^2)$ satisfy \eqref{Psi}.
Following the idea in \cite{Tsu-20}, we define
\begin{align}
    &u_N(x) := \frac{1}{\sqrt{\log N}} \sum_{j=10}^{N+10} \frac{1}{2^{\frac{2}{p}j^2}\sqrt{j}}\nabla^{\perp}\left(\psi(x) \cos (2^{j^2}x_1)\right),\\
    &v_N(x) := \frac{1}{\sqrt{\log N}} \sum_{j=10}^{N+10} \frac{1}{2^{\frac{2}{p'}j^2}\sqrt{j}}\nabla^{\perp}\left(\psi(x) \cos (2^{j^2}x_1)\right).
\end{align}
Then, since the support of the Fourier transform of $\nabla^{\perp}\left(\psi(x) \cos (2^{j^2}x_1)\right)$ is included in $\{\xi \in \mathbb{R}^2 \ ; \ 2^{j^2}-2 \leq |\xi| \leq 2^{j^2}+2\}$, we have
\begin{align}
    &
    \| u_N \|_{\dB_{p,\infty}^{\frac{2}{p}-1}}
    \leq
    \frac{C}{\sqrt{\log N}}
    \sup_{10 \leq j \leq N+10} \frac{1}{\sqrt{j}}
    \leq
    \frac{C}{\sqrt{\log N}}\to 0
    \qquad 
    {\rm as\ }N \to \infty,\\
    &
    \| v_N \|_{\dB_{p',\infty}^{\frac{2}{p'}-1}}
    \leq
    \frac{C}{\sqrt{\log N}}
    \sup_{10 \leq j \leq N+10} \frac{1}{\sqrt{j}}
    \leq
    \frac{C}{\sqrt{\log N}}\to 0
    \qquad 
    {\rm as\ }N \to \infty.
\end{align}
By the similar calculation as in the proof of Lemma \ref{lemm:direct}, we see that
\begin{align}
    &\Delta_k(-\Delta)^{-1}\mathbb{P}\div(u_N \otimes v_N)\\
    &\quad={}
    \frac{1}{2\log N}
    \sum_{j=10}^{N+10}
    \frac{1}{j}\Delta_k(-\Delta)^{-1}\mathbb{P}
    \begin{pmatrix}
        0 \\ \partial_{x_2}(\psi^2)
    \end{pmatrix}\\
    &\qquad
    +
    \frac{1}{2\log N}
    \sum_{j=10}^{N+10}
    2^{-2j^2}
    \frac{1}{j}\Delta_k(-\Delta)^{-1}\mathbb{P}\div( \nabla^{\perp}\psi \otimes \nabla^{\perp}\psi)
\end{align}
for all $k \leq -2$.
Then, by \eqref{stp2-1} and \eqref{stp2-2}, we obtain 
\begin{align}
    &\n{(-\Delta)^{-1}\mathbb{P}\div(u_N \otimes v_N)}_{\dB_{p,\infty}^{\frac{2}{p}-1}}\\
    &\quad 
    \geqslant{}
    \sup_{k\leq -2} 
    2^{(\frac{2}{p}-1)k}
    \n{\Delta_k(-\Delta)^{-1}\mathbb{P}\div(u_N \otimes v_N)}_{L^p}\\
    &\quad 
    \geqslant{}
    \frac{1}{2\log N}
    \sum_{j=10}^{N+10}
    \frac{1}{j}\sup_{k\leq -2} 
    2^{(\frac{2}{p}-1)k}
    \n{\Delta_k(-\Delta)^{-1}\mathbb{P}
    \begin{pmatrix}
        0 \\ \partial_{x_2}(\psi^2)
    \end{pmatrix} }_{L^p}\\
    &\qquad
    -
    \frac{1}{2\log N}
    \sum_{j=10}^{N+10}
    2^{-2j^2}
    \frac{1}{j}
    \n{(-\Delta)^{-1}\mathbb{P}\div( \nabla^{\perp} \psi \otimes  \nabla^{\perp} \psi)}_{\dB_{p,\infty}^{\frac{2}{p}-1}}\\
    &\quad 
    \geqslant{}
    \frac{c}{\log N}
    \sum_{j=10}^{N+10}
    \frac{1}{j}
    -
    \frac{C\| \psi \|_{\dot{H}^1}^2}{\log N},
\end{align}
which implies
\begin{align}
    \liminf_{N\to \infty}
    \n{(-\Delta)^{-1}\mathbb{P}\div(u_N \otimes v_N)}_{\dB_{p,\infty}^{\frac{2}{p}-1}}
    \geqslant 
    c.
\end{align}
Thus, we complete the proof.


\noindent
{\bf Conflict of interest statement.}\\
The author has declared no conflicts of interest.

\noindent
{\bf Acknowledgements.} \\
This work was supported by Grant-in-Aid for JSPS Research Fellow, Grant Number JP20J20941.

\end{document}